\documentclass[11pt,reqno]{amsart}

\usepackage{marginnote,amssymb,mathrsfs,latexsym,color,mathabx}
\usepackage[utf8]{inputenc}

\usepackage[margin=1.3in]{geometry}

\usepackage[colorlinks=true, linkcolor=blue]{hyperref}
\usepackage{mathtools}

\newcommand{\bbC}{{\mathbb C}}

\newcommand{\bbR}{{\mathbb R}}

\newcommand{\Ls}{\mathcal{L}}

\def\bbC{{\mathbb C}}

\def\bbR{{\mathbb R}}

\def\bX{{\mathbf X}}

\newcommand{\N}{\mathbb{N}}

\newcommand{\R}{\mathbb{R}}

\def\sf{{\mathfrak s}}
\newcommand{\cf}{\mathfrak{c}}
\newcommand{\tf}{\mathfrak{\tau}}

\def\Re{\operatorname{Re}}
\newcommand{\loc}{\mathrm{loc}}
\newcommand{\glob}{\mathrm{glob}}

\def\ov{\overline}

\def\e{\mathrm{e}}

\usepackage{changes}

\def\supp{\operatorname{supp}}

\numberwithin{equation}{section}

\DeclareFontFamily{U}{mathx}{\hyphenchar\font45}
\DeclareFontShape{U}{mathx}{m}{n}{
      <5> <6> <7> <8> <9> <10>
      <10.95> <12> <14.4> <17.28> <20.74> <24.88>
      mathx10
      }{}
\DeclareSymbolFont{mathx}{U}{mathx}{m}{n}
\DeclareFontSubstitution{U}{mathx}{m}{n}
\DeclareMathAccent{\widecheck}{0}{mathx}{"71}
\DeclareMathAccent{\wideparen}{0}{mathx}{"75}

\makeatletter
\newcommand{\leqnomode}{\tagsleft@true}
\newcommand{\reqnomode}{\tagsleft@false}
\makeatother

\newcommand{\dd}{\mathrm{d}}

\theoremstyle{theorem}
\newtheorem{theorem}{\sc \textbf{Theorem}}[section]  
\newtheorem{proposition}[theorem]{\sc \textbf{Proposition}}   
        
\newtheorem{lemma}[theorem]{\sc \textbf{Lemma}}

\theoremstyle{remark}

\newtheorem{remark}[theorem]{\sc \textbf{Remark}}
\begin{document}

\title[The Sobolev embedding constant]{The Sobolev embedding constant on Lie groups}  
 
\author[T.\ Bruno]{Tommaso Bruno}
\address{Department of Mathematics: Analysis, Logic and Discrete Mathematics, Ghent University,
Krijgslaan 281, 9000 Ghent, Belgium}
\email{tommaso.bruno@ugent.be}

\author[M.\ M.\ Peloso]{Marco M.\ Peloso}
\address{Dipartimento di Matematica, 
Universit\`a degli Studi di Milano, 
Via C.\ Saldini 50,  
20133 Milano, Italy}
\email{marco.peloso@unimi.it}

\author[M.\ Vallarino]{Maria Vallarino}
\address{Dipartimento di Scienze Matematiche ``Giuseppe Luigi Lagrange'',
  Politecnico di Torino, Corso Duca degli Abruzzi 24, 10129 Torino,
  Italy - Dipartimento di Eccellenza 2018-2022}
\email{maria.vallarino@polito.it}


\thanks{{\em Math Subject Classification} 26D10, 43A80, 46E35}

\thanks{All authors were  partially supported by the grant {\em Fractional Laplacians and subLaplacians on Lie groups and trees} of the Gruppo Nazionale per l'Analisi
  Matematica, la Probabilit\`a e le loro Applicazioni (GNAMPA) of the
  Istituto Nazionale di Alta Matematica (INdAM). T.\ Bruno gratefully acknowledges
  support by the Research Foundation – Flanders (FWO) through the
  postdoctoral grant 12ZW120N}  

\begin{abstract}

In this paper we estimate the Sobolev embedding
constant  on general noncompact Lie groups, for
sub-Riemannian inhomogeneous Sobolev spaces endowed with a left invariant
measure. The bound that we obtain, up to a constant depending only on the group and its sub-Riemannian
structure, reduces to the best known bound for
the classical inhomogeneous
Sobolev embedding constant on $\R^d$. As an application,  we prove 
 local and global Moser--Trudinger inequalities.
\end{abstract}

\maketitle

\section{Introduction}

The aim of this paper is to investigate the behaviour of the Sobolev
embedding constant in a sub-Riemannian setting, in particular on
noncommutative Lie groups. 

In the Euclidean space $\R^d$, if $\Delta$ denotes the classical {positive}
Laplacian and $\dot{L}^p_\alpha=\Delta^{\alpha/2}L^p$ the homogeneous
Sobolev space, it is well known that $\dot{L}^p_{\alpha}
\hookrightarrow L^q$ when $1<p<\infty$, $0\leq \alpha<d/p$ and 
$\frac{1}{q}=\frac{1}{p}-\frac{\alpha}{d}$. The best constant and the extremal functions for this
embedding have a long
history and  a multitude of applications, and  they can be obtained
from the analysis of the Hardy--Littlewood--Sobolev inequality. Lieb~\cite{Lieb} determined the best constant in the ``diagonal
case'' $p=q'$, and found an estimate in the other cases; see also earlier works by Aubin~\cite{Aubin} and Talenti~\cite{Talenti}.   If $L^p_\alpha=(I+\Delta)^{\alpha/2}L^p $
is  the inhomogeneous Sobolev space, then it is also well known that
$L^p_{\alpha}\hookrightarrow L^q$ when $1<p,q<\infty$, $0\leq \alpha<d/p$ and 
$\frac{1}{q} \geq \frac{1}{p}-\frac{\alpha}{d}$. The related best embedding
constant is not known, though it  can be 
bounded by the best constant for the embedding of homogeneous
spaces, up to  a dependence on the dimension $d$. 

\smallskip

On a general noncompact Lie group $G$, the natural substitutes of the
  Laplacian are sub-Laplacians with drift $\Ls$, see~\cite{BPTV},
  which are symmetric with respect to the left Haar measure
  $\lambda$. This setting, and this type of operators in particular, were
  studied in~\cite{HMM, Agrachev-et-al}, and an associated theory of
  Sobolev spaces, that we shall denote by $L^p_\alpha(\lambda)$,  was
  developed in~\cite{BPTV}. Since the Riesz transforms are not known
  to be bounded on $L^p$ when $1<p<\infty$ in such generality, while
  it is known that the appropriately shifted ones are bounded,
  see~\cite{BPTV}, it seems more natural to consider Sobolev spaces
  endowed with an inhomogeneous norm, which reduces to the Sobolev
  norm of $L^{p}_\alpha$ in the Euclidean case.

Our main result is an estimate for the constant of the embedding
$L^p_\alpha(\lambda) \hookrightarrow L^q(\lambda)$,  when
$1<p<\infty$, $0\leq \alpha<d/p$ and $\frac{1}{q}=\frac{1}{p}-\frac{\alpha}{d}$, of
the form $C\, S(p,q)$, where 
\begin{equation}\label{Spq}
S(p,q): = \min\bigg( \frac{q^{1/p'}}{p-1}, \frac{{p'}^{1/q}}{q'-1}\bigg)
\end{equation}
and $C$ depends only
on the group and its chosen sub-Riemannian structure.   Here and
throughout the paper, given any 
$p\in(1,\infty)$ we denote by $p'$ its conjugate exponent, that
is, $p'=p/(p-1)$. In terms of the dependence on $p$ and $q$, such a
bound is comparable to the best known bound in $\R^d$ for the Sobolev
embedding constant for inhomogeneous spaces associated with the
Laplacian, while it is new in noncommutative groups. In addition to this, we shall also
discuss the more general case of relatively invariant measures where,
despite the Sobolev embeddings in general fail~\cite{BPTV}, we are
able to prove alternative results. 

A well-established application of the Sobolev embedding theorem, both
in the homogeneous and inhomogeneous case, is the classical
Moser--Trudinger inequality~\cite{Trudinger, Moser},  
which arises as a substitute of boundedness for
functions in the Sobolev space $L^{p}_{d/p}$, as this does not embed in
$L^\infty$.  By means of our  quantitative Sobolev 
embedding, we prove quantitative versions of local and global
Moser--Trudinger inequalities. Our approach is close in spirit, and
inspired by, \cite{Oz}. We refer the reader also to 
the recent work~\cite{RuzYess}.  

\smallskip

The analysis of sub-Laplacians and more generally of subelliptic
differential operators has attracted a great deal of attention since
their appearance in the study of Kohn-Laplacians and the renowned
sum-of-squares theorem of H\"ormander.  It appears then very natural to 
 extend geometric
 and functional inequalities from the Euclidean, elliptic case
 to a subelliptic setting, also in a quantitative
 form. Earlier breakthroughs were, e.g., Sobolev embeddings on stratified Lie groups~\cite{Folland} and the Poincar\'e inequality for sums of squares on $\R^d$~\cite{Jerison}. Among more recent works, we mention the Sobolev embedding theorem  on
 unimodular Lie groups~\cite{CRTN}, a lower bound for the Hausdorff--Young constant
on general Lie groups~\cite{CMMP}, the best constants for Sobolev
and Gagliardo--Nirenberg inequalities on graded groups~\cite{RuzYess}, and Poincar\'e inequalities on Lie groups~\cite{RS,BPV3}. This paper fits into this order of ideas and line of
research; we refer the reader also to~\cite{FR, RTY, BPTV} and the
references therein. We emphasize that our setting is a general (connected) Lie
group, endowed with a left Haar measure which,
in general, has exponential volume growth and is non-doubling.

\smallskip

The structure of the paper is as follows. In Section~\ref{s: setting}, we describe the setting and all the preliminary results we shall need. Section~\ref{Sec_SEC} is the core of the paper, and contains the proof of the quantitative Sobolev embedding, whose constant is compared in Section~\ref{sec:Eucl} with the Euclidean ones. In Section~\ref{Sec_MT} we prove a quantitative Moser--Trudinger inequality, and in Section~\ref{Sec_gen} we discuss the case of more general measures.

\subsection*{Acknowledgements} We thank the anonymous referees for carefully reading the manuscript and making a number of suggestions and comments that led us to improve the clarity of our presentation. 

\section{Setting and Preliminaries}\label{s: setting}

Let $G$  be a noncompact connected Lie group with identity $e$. Let $\lambda$ be a left Haar measure on $G$, and $\delta$ be the modular function.

Let $\mathbf{X}= \{ X_1,\dots, X_\ell\}$ be a family of left-invariant
linearly independent vector fields which satisfy H\"ormander's
condition. Let $d_C(\, \cdot\, ,\,  \cdot\, )$ be its associated
left-invariant Carnot--Carathéodory distance. We let $|x|=d_C(x,e)$,
and denote by $B_r$ the ball centred at $e$ of radius $r$. We denote by $V(r)=\lambda(B_r)$ the measure of
of the ball $B_r$ with respect to $\lambda$. We recall
(cf.~\cite{Guiv,Varopoulos1}) that there exist two constants $d\in
\N^*$ and $D>0$ such that  
\begin{equation}\label{pallepiccolegrandi}
\begin{split}
C^{-1} r^d \leq V(r) \leq C r^d\quad \forall r\in (0,1], \qquad 
 V(r)\leq C \e^{Dr}\quad \; \, \forall r\in (1,\infty),
 \end{split}
\end{equation}
where $C>0$ is independent of $r$. We emphasize that $d$ is uniquely determined by $G$ and $\mathbf{X}$, while the set of $D>0$ such that~\eqref{pallepiccolegrandi} holds is independent of $\mathbf{X}$ but does not have a minimum in general; consider, e.g., the case when $G$ has polynomial growth. From this point on, we fix a $D>0$ for which~\eqref{pallepiccolegrandi} holds, and observe that the metric measure space $(G, d_C,\lambda)$ is locally
doubling, but not doubling in general. 

If $p\in [1,\infty)$, the spaces of (equivalent classes of) measurable
functions whose $p$-power is integrable with respect to $\lambda$
will be denoted by $L^p(\lambda)$, or simply $L^p$, and endowed with the usual norm
which we shall denote by $\| \cdot \|_{L^p(\lambda)}$. The space
$L^\infty$ is defined analogously.
The convolution between two functions $f$ and $g$, when it exists, is defined by 
\[
f*g(x) =\int_G f(xy)g(y^{-1})\, \dd \lambda(y),\qquad  x\in G\,.
\]
We recall Young's inequality, which has the following form~\cite{HR}: if $1\leq p\leq q\leq \infty$ and $r\geq 1$ is such that $\frac{1}{p}+ \frac{1}{r}=1+\frac{1}{q}$, then
\begin{equation}\label{Young}
\begin{split}
  \|f*g\|_{L^q(\lambda)} &\leq \|f\|_{L^p(\lambda)}
   \| \widecheck{g}\|_{L^r(\lambda)}^{r/p'} \|g\|_{L^r(\lambda)}^{r/q}, \qquad  (q<\infty)\\
\|f*g\|_{L^\infty} & \leq \|f\|_{L^p(\lambda)}  \| \widecheck{g}\|_{L^{p'}(\lambda)},
   \end{split}
\end{equation}
where $\widecheck{g}(x)= g(x^{-1})$. We denote by $\Ls$ the intrinsic sub-Laplacian on $G$ associated with $\bX$, see~\cite{Agrachev-et-al}, 
\begin{equation*}
\Ls = - \sum_{j=1}^\ell (X_j^2 +(X_j\delta)(e) X_j),
\end{equation*}
which is symmetric on $L^2(\lambda)$, and essentially self-adjoint on $C_c^\infty(G)$, see~\cite{HMM}. We shall denote by $\Ls$ as well its unique self-adjoint extension.

The operator $\Ls$ generates a diffusion semigroup, i.e.\
$(\e^{-t\Ls})_{t>0}$ extends to a contraction semigroup on
$L^p(\lambda)$ for every $p \in [1, \infty]$ (see~\cite{HMM}) whose
infinitesimal generator, with a slight abuse of notation, we still
denote by $\Ls$. We denote by $p_t^\delta$
the convolution kernel of $\e^{-t\Ls}$, and we recall that by
~\cite[Theorems VIII.2.9, VIII.4.3 and IX.1.3]{VCS} there exist
constants $b, c>0$ depending only on $G$ and $\bf{X}$ such that 
\begin{equation}\label{heatkernelestimate}
p_t^\delta(x)  \leq c \, (1\wedge t)^{-\frac{d}{2}} \, \e^{-\frac{1}{4} t \cf(\delta)^2}\,  \e^{-b \frac{|x|^2}{t}}, \qquad x\in G,\,t>0,
\end{equation}
where $ \cf(\delta) = ( |X_1\delta(e)|^2 + \cdots +
|X_\ell\delta(e)|^2)^{1/2}$. Let $b_0 = \sqrt{b}/2$, and define
\begin{equation}\label{fixedtranslation}
\tf_\delta = \max \left\{\frac{2}{b} \left[2D+b_0 \right]^2 - \frac{1}{4}\cf(\delta)^2, 1\right\}.
\end{equation}

Following~\cite{BPTV}, when $p\in (1,\infty)$ and $\alpha> 0$ we
define the Sobolev spaces $L^p_\alpha(\lambda)$ as the set of
functions $f\in L^p(\lambda)$ such that $(\tau_\delta I +
\Ls)^{\alpha/2} f\in L^p(\lambda)$, endowed with the norm 
\begin{equation}\label{equivtranslation}
\| f\|_{L^p_{\alpha}(\lambda)} = \|(\tau_\delta I + \Ls)^{\alpha/2}f\|_{L^p(\lambda)}.
\end{equation}
If $\alpha=0$, we let $L^p_0(\lambda)=L^p(\lambda)$. We recall
that~\eqref{equivtranslation} is equivalent to the norm
$\|f\|_{L^p(\lambda)} + \|\Ls^{\alpha/2}
f\|_{L^p(\lambda)}$, see~\cite{BPTV}. The reason for choosing the shift $\tau_\delta$ in the definition of $L^p_\alpha(\lambda)$ will be clarified later on; for more details about $\tau_\delta$, we refer the reader to the beginning of Section~\ref{sec:Eucl} below.

In~\cite{BPTV} the Sobolev embeddings $L^p_\alpha(\lambda) \hookrightarrow L^q(\lambda)$ when $0<\alpha<d/p$ and $q> p$ are such that $\frac{1}{q}=
\frac{1}{p}-\frac{\alpha}{d}$, were established.  In this paper we find an explicit bound for the
embedding constants, in the spirit which we now explain.

Throughout the paper, we shall disregard any dependence of the
embedding constants on $G$ and $\bf{X}$, which are assumed to be fixed
once and for all from this point on. We shall, instead, obtain
explicit results in terms of the dependence on $p$, $q$ and
$\alpha$. A generic constant depending only on $G$ and $\bf{X}$ will
be denoted by $C$ or $C(G,\bf{X})$, and its value may vary from line
to line. Recall in particular that $d=d(G,\bf{X})$.

For $\alpha>0$, let $G_{\delta}^{\alpha}$ be the convolution kernel of
$(\tau_\delta I +\Ls)^{-\alpha/2}$. Let 
\begin{equation}\label{Glocglob}
  G_{\delta}^{\alpha, \loc} =G_{\delta}^{\alpha}\mathbf{1}_{B_1},
  \qquad G_{\delta}^{\alpha, \glob} = G_{\delta}^{\alpha}\mathbf{1}_{B_1^c}.
\end{equation}
The following is a refined version of~\cite[Lemma 4.1]{BPTV}.

\begin{lemma}\label{Lemma4.1-revised}
  There exists $C=C(G, \mathbf{X})>0$ such
  that, for $\alpha\in (0,d)$ and $x\in G$,
\begin{align*}
|G_{\delta}^{\alpha, \loc} (x)| &\leq C\, \frac{\alpha}{d-\alpha} |x|^{\alpha-d}\mathbf{1}_{B_1}(x) ,\\
|G_{\delta}^{\alpha, \glob}(x)| &\leq  C \, \e^{-(2D+b_0)|x|}\mathbf{1}_{B_1^c}(x).
\end{align*}
  \end{lemma}
\begin{proof}
  We recall that the convolution kernel $G_{\delta}^{\alpha}$ can be written as
  $$
  G_{\delta}^{\alpha}=\frac{1}{\Gamma(\alpha/2)} \int_0^\infty 
t^{\alpha/2-1} \e^{-\tau_\delta \, t} p_t^\delta\, \dd t,
  $$
so that by~\eqref{heatkernelestimate}
\[
G_{\delta}^{\alpha}(x) \le \frac{C}{\Gamma(\alpha/2)} \int_0^\infty 
t^{\alpha/2-1}(1\wedge t)^{-d/2} \e^{-(\tau_\delta +\frac{1}{4} \cf(\delta)^2) t} \e^{-b|x|^2/t}\,
\dd t \,.
\]
Set $a=\tau_\delta +\frac{1}{4} \cf(\delta)^2$. Since $at +b|x|^2/t \geq \frac{1}{2} ( at + b|x|^2/t +\sqrt{2ab}|x|)$, we  see that when $|x|\ge1$,
\begin{align*}
G_{\delta}^{\alpha}(x)
& \le \frac{C}{\Gamma(\alpha/2)}  
\e^{-\frac{1}{2} \sqrt{2ab} |x|} \int_0^\infty t^{\alpha/2 -1} ( 1 \wedge t )^{-d/2}
\e^{-\frac{at}{2}-\frac{b}{2t}}\, \dd t \le C \, \e^{-(2D+b_0)|x|} \,.
\end{align*}
On the other hand, when $|x|\le 1$, splitting the integral we have 
\begin{align*}
 G_{\delta}^{\alpha}(x) 
  &\le C\, \alpha \bigg(
  \int_0^1 t^{(\alpha-d)/2 -1} \e^{-b|x|^2/t} \,
  \dd t  +  \int_1^\infty  t^{\alpha/2 -1}  \e^{-a t } \, \dd t \bigg)\\ 
  &  \eqqcolon{C} \, \alpha\, \left(
G_1(x) + G_2(x) \right) .
\end{align*}
It is clear, since $\alpha\in (0,d)$ and $a\geq 1$, that $G_2(x)\le C$.
Since $\alpha\in (0,d)$, we also have
\begin{align*}
  G_1(x)
  & =
 |x|^{\alpha -d}\bigg(  \int_{|x|^2}^1+ \int_1^\infty \bigg) u^{(d-\alpha)/2-1} \e^{-bu} \, \dd
    u \le C\,
 |x|^{\alpha -d}\bigg(   \frac{1}{d-\alpha} (1-|x|^{d-\alpha})+1  \bigg)\,,
\end{align*}
and the conclusion follows.
\end{proof}

\section{The Sobolev embedding constant}\label{Sec_SEC}

We are now ready to state our main result.  Recall that the constant $S(p,q)$ is defined in \eqref{Spq}.

\begin{theorem}\label{teo:embed}
Let $p\in (1,\infty)$, $\alpha \in [0,d/p)$ and $q\in [p,\infty)$ be such that $\frac{1}{q}= \frac{1}{p}-\frac{\alpha}{d}$. Then there exists $A_1=A_1(G,\mathbf{X})>0$ such that for all $f\in L^{p}_{\alpha}(\lambda)$
\[
 \| f\|_{L^q(\lambda)}\leq A_1\, S(p,q) \|f\|_{L^p_\alpha(\lambda)}.
\]
\end{theorem}

\begin{proof}
When $\alpha=0$ and hence $q=p$, the statement is the trivial embedding $L^p \hookrightarrow L^p$. Since the function $x \mapsto x^{1-1/x}/(x-1)$ is bounded from below for $x>1$, one sees that $S(p,p) \geq 1/c$ for some $c>0$. Then
\[
 \| f\|_{L^p(\lambda)} =\|f\|_{L^p_0(\lambda)}  \leq c \, S(p,p) \|f\|_{L^p_0(\lambda)},
\]
and from this point on we may then assume $\alpha>0$ and $q>p$. Define
\[
K_\alpha(x) =  |x|^{\alpha-d} \mathbf{1}_{B_1}(x) , \qquad \tilde{K}_\alpha(x) = \e^{-(2D+b_0)|x|}\mathbf{1}_{B_1^c}(x).
\]
We claim that 
\begin{align}
\| f\ast K_\alpha \|_{L^q(\lambda)}
&\leq C(G,\mathbf{X})  \, \frac{d-\alpha}{\alpha}
    \,  \frac{{q^{1/p'}}}{p-1}  \|f\|_{L^p(\lambda)},\label{claim}\\
\| f\ast \tilde{K}_\alpha\|_{L^q(\lambda)} &\leq C(G,\mathbf{X})  \|f\|_{L^p(\lambda)}.\label{claim2}
\end{align}
By combining these bounds and Lemma~\ref{Lemma4.1-revised}, we obtain that
\begin{equation}\label{first-half}
\| (\tau_\delta I  +\Ls )^{-\alpha/2} f\|_{L^{q}(\lambda)}
\leq  A_1(G,\mathbf{X})  \frac{{q^{1/p'}}}{p-1} \|f\|_{L^{p}(\lambda)}.
\end{equation}
Observe that $q^{1/p'}/(p-1)$ is bounded away from zero when $q\geq p>1$.
Assuming the claims for a moment, we complete the proof.
Observe that the condition $\frac1q=\frac1p-\frac{\alpha}{d}$ is
invariant under the involution $(p,q)\mapsto(q',p')$. Set $Q(p,q)=
\frac{q^{1/p'}}{p-1}$. By duality, from~\eqref{first-half} we have
$$
\| (\tau_\delta I  +\Ls )^{-\alpha/2} f\|_{L^{p'}(\lambda)}
\leq A_1 Q(p,q) \|f\|_{L^{q'}(\lambda)},
$$
that is, switching the roles of the pairs $(p,q)$ and $(q',p')$,
$$
\| (\tau_\delta I  +\Ls )^{-\alpha/2} f\|_{L^{q}(\lambda)}
\leq A_1 Q(q',p') \|f\|_{L^{p}(\lambda)}.
$$
This inequality, together with~\eqref{first-half} gives
$$
\| (\tau_\delta I  +\Ls )^{-\alpha/2} f\|_{L^{q}(\lambda)}
\leq A_1 \min\big(Q(p,q),  Q(q',p') \big) \|f\|_{L^{p}(\lambda)},
$$
which implies
$$
\|f\|_{L^{q}(\lambda)}\leq A_1 \, S(p,q) \|f\|_{L^{p}_\alpha(\lambda)}.
$$

Thus, it remains to prove the claims. 
The bound~\eqref{claim2} follows by observing that $\tilde K_\alpha =({\tilde K}_\alpha)^{\widecheck{\,}} $ and by applying Young's inequality ~\eqref{Young}
\begin{equation}\label{tilde K}
\|f\ast \tilde K_{\alpha}\|_{L^q(\lambda)}\leq \|f\|_{L^p(\lambda)} \|\tilde K_{\alpha}\|_{L^r(\lambda)}^{r(1/p'+1/q)},
\end{equation}
where $r\in (1,\infty)$ is such that
$\frac1p+\frac1r=1+\frac1q$. We then have 
$$
\begin{aligned}    
  \| \tilde K_{\alpha} \|^r_{L^r(\lambda)}   
  &\leq C  \int_{B_1^c} \e^{-r(2D+ b_0)|x|} \,\dd\lambda(x)   \\
      & \le C\,   \sum_{k=0}^\infty \int_{2^{k}\leq |x|< 2^{k+1}}  \e^{-r(2D+ b_0)|x|} \,\dd\lambda(x)    \le C\,   \sum_{k=0}^\infty  \e^{-r(2D+b_0) 2^k+D2^{k+1}}  \leq C ,
  \end{aligned}
  $$
  which combined with \eqref{tilde K} implies \eqref{claim2}. The remainder of the proof will be devoted to
show~\eqref{claim}. 

\smallskip

For $0<s\leq 1$, define $K_{\alpha,s}^{(1)} = K_\alpha\, \mathbf{1}_{B_s}$ and $K_{\alpha,s}^{(2)} = K_\alpha\, \mathbf{1}_{B_s^c}$. Notice that $K_{\alpha,s}^{(1)} = \widecheck{K}_{\alpha,s}^{(1)} $ and that the same holds for $K_{\alpha,s}^{(2)}$. Let now $\tilde p\in (1,\infty)$ and $\tilde q\in (\tilde p,\infty)$ be such that $\frac{1}{\tilde q}= \frac{1}{\tilde
  p}-\frac{\alpha}{d}$, and observe that
\begin{equation}\label{tildepqprop}
(\alpha-d)\tilde p' +d = -\frac{d \tilde{p}'}{\tilde q}, \qquad \frac{\tilde p}{\tilde q} = 1-\tilde p\frac{\alpha}{d}, \qquad \frac{1}{\tilde p'}\Big(1- \frac{\tilde p}{\tilde q}\Big) = (\tilde p-1)\frac{\alpha}{d}.
\end{equation}

    By Young's inequality~\eqref{Young}, there
exists $C>0$ depending only on $G$ and $\bf X$ such that 
\begin{equation}\label{bound1}
\| f \ast K_{\alpha,s}^{(1)}\|_{L^{\tilde p}(\lambda)}  \leq
\|f\|_{L^{\tilde p}(\lambda)}
\|{K}_{\alpha,s}^{(1)}\|_{L^{1}(\lambda)}^{1/\tilde p} \|
  \widecheck{K}_{\alpha,s}^{(1)}\|_{L^{1}(\lambda)}^{1/\tilde p'}  \leq C\frac{1}{\alpha} s^\alpha
\|f\|_{L^{\tilde p}(\lambda)} 
\end{equation}
and 
\begin{align}\label{bound02}
\| f \ast K_{\alpha,s}^{(2)}\|_{L^\infty} 
\leq \|f\|_{L^{\tilde p}(\lambda)}
   \|\widecheck{K}_{\alpha,s}^{(2)}\|_{L^{\tilde p'}\!(\lambda)} \leq C
 \left(\frac{\tilde q}{d\tilde p'}\right)^{1/\tilde p'} (s^{-d \tilde{p}'/\tilde q}-1)^{1/\tilde p'}\|f\|_{L^{\tilde p}(\lambda)} .
\end{align}
For  $t>0$ we now set 
\[
s(t) = \left[  1+ \frac{d\tilde p'}{\tilde q} \left( \frac{t}{2} \right)^{\tilde p'}\right]^{{-\frac{\tilde q}{d \tilde{p}'}}},
\]
and observe that $s(t)\leq 1$ for every $t>0$. By~\eqref{bound02},
\begin{equation}\label{bound2}
\| f \ast K_{\alpha,s(t)}^{(2)}\|_{L^\infty} \leq C  \frac{t}{2} \|f\|_{L^{\tilde p}(\lambda)} \qquad \forall t>0\,.
\end{equation}
Thus, with $C$ the same constant as in~\eqref{bound1} and~\eqref{bound02},
\begin{align*}
\sup_{t>0} t\, \lambda\! &\left(\left\{ x\colon \:  |f* K_\alpha(x)|>t\right\} \right)^{1/\tilde q}\\
& = C \|f\|_{L^{\tilde p}(\lambda)} \: \sup_{t>0} t\, \lambda\!\left(\left\{ x\colon \:  |f\ast K_\alpha(x)|>Ct \|f\|_{L^{\tilde p}(\lambda)}\right\} \right)^{1/\tilde q} \\
& \leq C\|f\|_{L^{\tilde p}(\lambda)}  \: \sup_{t>0} t\, \lambda\!\left(\left\{ x\colon \:  |f\ast K_{\alpha,s(t)}^{(1)}(x)|>C\frac{t}{2} \|f\|_{L^{\tilde p}(\lambda)}\right\} \right)^{1/\tilde q} \\
& \qquad \qquad +C\|f\|_{L^{\tilde p}(\lambda)} \:\sup_{t>0} t\, \lambda\!\left(\left\{ x\colon \:  |f\ast K_{\alpha,s(t)}^{(2)}(x)|>C\frac{t}{2} \|f\|_{L^{\tilde p}(\lambda)}\right\} \right)^{1/\tilde q}\\
& = C \|f\|_{L^{\tilde p}(\lambda)}\: \sup_{t>0} t\, \lambda\!\left(\left\{ x\colon \:  |f\ast K_{\alpha,s(t)}^{(1)}(x)|>C\frac{t}{2} \|f\|_{L^{\tilde p}(\lambda)}\right\} \right)^{1/\tilde q},
\end{align*}
since $s(t)$ was chosen so that the second super-level set was empty. By~\eqref{bound1}, we get
\begin{align*}
\sup_{t>0} t\, \lambda\! & \left(\left\{ x\colon \:  |f\ast K_{\alpha,s(t)}^{(1)}(x)|>C\frac{t}{2} \|f\|_{L^{\tilde p}(\lambda)}\right\} \right)^{1/\tilde q}\\
&\leq \sup_{t>0} t\, \left[ \left( \frac{2}{ {C} t \|f\|_{L^{\tilde p}(\lambda)}} \right)^{\tilde p} \| f\ast K_{\alpha,s(t)}^{(1)}\|_{L^{\tilde p}(\lambda)}^{\tilde p} \right]^{1/\tilde q}\\
& \leq \sup_{t>0} t \left( \frac{ {C} t\|f\|_{L^{\tilde p}(\lambda)} }{2} \right)^{-\tilde p/\tilde q} \left( \frac{s(t)^\alpha}{\alpha}\right)^{\tilde p/\tilde q}  {C^{\tilde p/\tilde{q}}}\|f\|_{L^{\tilde p}(\lambda)}^{\tilde p/\tilde q}\\
 & =  \left( \frac{2}{\alpha} \right)^{\tilde p/ \tilde q}
 \: \sup_{t>0} t^{1-\tilde p/\tilde q} \left[1 + \frac{d \tilde p'}{\tilde q}
   \left( \frac{t}{2}\right)^{\tilde p'} \right]^{-\frac{1}{\tilde p'}(1- \frac{\tilde p}{\tilde q})} \\
& = \frac{2}{\alpha^{\tilde p/\tilde q}} 
\left( \frac{\tilde q}{d\tilde p'} \right)^{\frac{1}{\tilde p'}(1- \frac{\tilde p}{\tilde q})} 
 \,\sup_{u>0} u^{1-\tilde p/\tilde q} (1 + u^{\tilde p'})^{-\frac{1}{\tilde p'}(1- \frac{\tilde p}{\tilde q})}. 
\end{align*}
It is now easy to see that, for every $\tilde p$ and $\tilde q$,
\[
\sup_{u>0} u^{1-\tilde p/\tilde q} (1 + u^{\tilde
  p'})^{-\frac{1}{\tilde p'}(1- \frac{\tilde p}{\tilde q})}
=\sup_{v>0} \big[ v/(1+v)\big]
^{\frac{1}{\tilde p'}(1- \frac{\tilde p}{\tilde q})} =1.
\]
Moreover,  by~\eqref{tildepqprop} we end up with the inequality
\begin{align}
  \|f\ast K_\alpha\|_{L^{\tilde q,\infty}(\lambda)}
&   = \sup_{t>0} t\, \lambda\! \left(
    \left\{ x\colon \: |f\ast K_\alpha(x)|>t\right\}
 \right)^{\frac{1}{\tilde q}} \notag \\
 & \leq C \alpha^{\tilde p \alpha/d -1} \left( \frac{\tilde q}{d\tilde p'}
  \right)^{(\tilde p -1)\alpha/d} 
  \|f\|_{L^{\tilde p}(\lambda)}. \label{debtildeptildeq}
\end{align}
In other words, the operator defined by $\mathcal{K}_\alpha f = f* K_\alpha$ is
of weak type $(\tilde{p}, \tilde q)$ for every $\tilde p,  \tilde q$
such that $\frac{1}{\tilde q}= \frac{1}{\tilde p}-\frac{\alpha}{d}$,
$1<  \tilde p < \tilde q <\infty$, $0< \alpha< d$.  

In a similar way we can also prove that $\mathcal{K}_\alpha$ is
of weak type $(1, \tilde q)$ for $\frac{1}{\tilde q}= 1-\frac{\alpha}{d}$ and $0< \alpha< d$. Indeed, the estimate \eqref{bound1} holds also for $\tilde p=1$ and 
\begin{equation}\label{K2bis}
\| f \ast K_{\alpha,s}^{(2)}\|_{L^\infty} \leq C \|f\|_{L^{1}(\lambda)} \times
\begin{cases}
s^{\alpha-d}& \mbox{if } s<1 \\
0 &   \mbox{if } s\geq 1.
\end{cases}
\end{equation}
We now set 
\[
s(t) = \begin{cases}
\left(1+   \frac{t}{2} \ \right)^{1/(\alpha-d)}&t\geq 2\\
1&0<t<2\,,
\end{cases}
\]
which is $\leq 1$. Then \eqref{bound2} holds also in this case and we obtain as above that
\begin{align*}
\sup_{t>0} t\, \lambda\! &\left(\left\{ x\colon \:  |f* K_\alpha(x)|>t\right\} \right)^{1/\tilde q}\\
 & \leq  C \|f\|_{L^{1}(\lambda)}\: \sup_{t>0} t\, \lambda\!\left(\left\{ x\colon \:  |f\ast K_{\alpha,s(t)}^{(1)}(x)|>C\frac{t}{2} \|f\|_{L^{\tilde p}(\lambda)}\right\} \right)^{1/\tilde q}\\
& \leq   C \|f\|_{L^{1}(\lambda)}\: \sup_{t>0} t\, \left(  \frac{2}{ {C} t \|f\|_{L^{1}(\lambda)}}   \| f\ast K_{\alpha,s(t)}^{(1)}\|_{L^{1}(\lambda)} \right)^{1/\tilde q}\,.
\end{align*}
We now notice that 
\begin{align*}
  \sup_{0<t<2} t\, \left(  \frac{2}{ {C} t \|f\|_{L^{1}(\lambda)}}   \| f\ast K_{\alpha,s(t)}^{(1)}\|_{L^{1}(\lambda)} \right)^{1/\tilde q}
& \leq  \sup_{0<t<2} t \left( \frac{t\|f\|_{L^{1}(\lambda)} }{2} \right)^{-1/\tilde q} \left( \frac{1}{\alpha}\right)^{1/\tilde q} \|f\|_{L^1(\lambda)}^{1/\tilde q}\\
&= 2\alpha^{-1/{\tilde q}}\,,
 \end{align*}
while
\begin{align*}
\sup_{t\geq 2} t\, \left(  \frac{2}{ {C} t \|f\|_{L^{1}(\lambda)}}   \| f\ast K_{\alpha,s(t)}^{(1)}\|_{L^{1}(\lambda)} \right)^{1/\tilde q}& \leq  \sup_{t\geq 2} t \left( \frac{t\|f\|_{L^{1}(\lambda)} }{2} \right)^{-1/{\tilde q}} \left( \frac{s(t)^{\alpha}}{\alpha}\right)^{1/\tilde q} \|f\|_{L^1(\lambda)}^{1/\tilde q}\\
&\leq C\sup_{t\geq 2} t^{1-\frac{1}{\tilde q}}\left(  \frac{2}{\alpha} \right)^{1/{\tilde q}}\left(\frac{t}{2}\right)^{-1/d}=C\,\alpha^{-1/{\tilde q}}\,.
\end{align*}
This proves that 
\begin{equation}\label{deb1tildeq}
\|f\ast K_\alpha\|_{L^{\tilde q,\infty}(\lambda)}   \leq C \alpha^{-1/\tilde q}\|f\|_{L^{1}(\lambda)}.
\end{equation}
We shall now use the Marcinkiewicz interpolation theorem
 for two
specific choices of the couple $(\tilde{p}, \tilde q)$. Being $p\in
(1,\infty)$, $q \in (p, \infty)$, and $\alpha/d=1/p-1/q$ as in the
statement, we define 
\begin{equation}\label{p1q1p2q2}
\left
  (\frac{1}{p_{1}},\frac{1}{q_{1}}\right)=\left(1,1-\frac{\alpha}{d}\right),
\qquad \left(\frac{1}{p_{2}},
  \frac{1}{q_{2}}\right)=\left(\frac{\alpha}{d}+\frac{1}{q+1},\frac{1}{q+1}\right). 
\end{equation}
By the above, $\mathcal{K}_{\alpha}$ is both of weak type 
$(1,q_{1})$ and $(p_{2},q_{2})$  with norms $M(1,q_1)$ and
$M(p_2,q_2)$ respectively, given by
\begin{align*}
 M(1,q_1) & = \alpha ^{-(1-\alpha/d)} , \label{M1} \\
 M(p_2,q_2)  & =\Big(\frac{d^{\alpha/d}}{\alpha}\Big)
  \left( \frac{\alpha}{d}\right)^{\frac{\alpha/d}{\alpha/d+1/(q+1)}}
\Big[ \Big(1-\frac{\alpha}{d} -\frac{1}{q+1}
   \Big)(q+1)\Big]^{\frac{1}{1+d/(\alpha(q+1))}- \frac{\alpha}{d}}.
  \end{align*}
We select
\begin{equation*}\label{teta}
\theta= \frac{1-\frac{1}{p}}{1-\frac{\alpha}{d}
  -\frac{1}{q+1} } .
\end{equation*}
Notice that we indeed
have $0<\theta<1$, $1/p=(1-\theta)/p_{1}+\theta/p_{2}$ and
  $1/q=(1-\theta)/q_{1}+\theta/q_{2}$. Thus, 
$\mathcal{K}_\alpha$ is of
strong type $(p,q)$, i.e.\ bounded from $L^p(\lambda)$ to
$L^q(\lambda)$, with norm bounded by
\begin{equation*}
  C  M_0(1, q_1, p_2, q_2)^{1/q} 
M(1,q_1)^{1-\theta} M(p_2,q_2)^{\theta},
\end{equation*}
see e.g.~\cite[Ch.\ XII, (4.18)]{Z}, where
\begin{align*}
M_0(1, q_1, p_2, q_2) = \frac{q(p_2/p)^{q_2/p_2}}{q_2-q} + \frac{q/p^{q_1}}{q-q_1}  \label{M0}
   . \end{align*}
If we observe that
\begin{equation}\label{Ipq}
M_0(1, q_1, p_2, q_2)^{1/q} 
M(1,q_1)^{1-\theta} M(p_2,q_2)^{\theta} \leq C\, \frac{d-\alpha}{\alpha} \,  \frac{q^{1/p'}}{{p-1}},
\end{equation}
then we get~\eqref{claim}, which concludes the proof of the theorem.

We now prove~\eqref{Ipq}. First we consider $M_1=M(1,q_1)$, and simply observe that
\[
M_1  = \alpha^{-1} d^{\alpha/d}(\alpha/d)^{\alpha/d} \leq  d \, \alpha^{-1}
\]
as $\alpha/d\leq 1$ and $x^x \leq 1$ for $x\in (0,1]$.

Then we consider $M_0=M_0(1, q_1, p_2, q_2)$, and define
\[
C(p,q) = p^{-p'q/(q+p')} \Big(1+\frac{p'}{q}\Big), \qquad y = \frac{\alpha}{d}(q+1).
\]
Since
\[
\frac{p_2}{q_2}= 1+y, \qquad \frac{p_2}{p} = y+1+\frac{1}{q},
\]
we get
\[
M_0 =  q   \left(y+1+\frac{1}{q}\right)^{1+y} (1+y)^{-(1+y)} + C(p,q). 
\]
Moreover
\begin{align*}
 \left(y+1+\frac{1}{q}\right)^{1+y} (1+y)^{-(1+y)}= \left[ \left( 1+ \frac{1}{q(1+y)}\right)^{q(1+y)}\right]^{1/q}  \leq \e
\end{align*}
since $q(1+y) \geq 1$ and by the estimate $ (1 +\frac{1}{x})^x \leq \e$ for $x\geq 1$.
 Thus  $M_0 \leq   \e \, q + C(p,q)$. 
 
 We then consider $M_2=M(p_2,q_2)$, and estimate $M_2^\theta$. We first observe that
\[
M_2^\theta \leq d^\theta \alpha^{-\theta} \left( \frac{\alpha}{d}\right)^{\theta \frac{\alpha/d}{\alpha/d+1/(q+1)}} \Big[ \Big(1-\frac{\alpha}{d} -\frac{1}{q+1}  \Big)(q+1)\Big]^{\theta\frac{\alpha/d}{\alpha/d+1/(q+1)}- \theta\frac{\alpha}{d}} 
\]
and that
\begin{multline}
\left( \frac{\alpha}{d}\right)^{\theta \frac{\alpha/d}{\alpha/d+1/(q+1)}} \Big[ \Big(1-\frac{\alpha}{d} -\frac{1}{q+1}  \Big)(q+1)\Big]^{\theta\frac{\alpha/d}{\alpha/d+1/(q+1)}- \theta\frac{\alpha}{d}} \\ = \left[ \left(\frac{\alpha}{d}\right)^{\frac{1}{1-z} }(q+1)\right]^{ (1-1/p) \frac{\alpha/d}{z} } (1-z )^{ (1-1/p) \frac{\alpha/d}{z} } \label{MM2}
\end{multline}
where $z= \frac{\alpha}{d} + \frac{1}{q+1}$. Observe that $0< z < 1/p<1 $, and that
\begin{equation}\label{alphadqz}
\frac{\alpha/d}{\alpha/d + 1/(q+1)} = \frac{\alpha}{dz} = \frac{(q-p)(q+1)}{q(q+1)-p} \leq 1.
\end{equation}
Therefore
\[
\left(\frac{\alpha}{d}\right)^{\frac{1}{1-z} } \leq \frac{\alpha}{d}, \qquad (1-z )^{ (1-1/p) \frac{\alpha/d}{z} } \leq 1.
\]
Observe now that 
\begin{align*}
\left[ \left(\frac{\alpha}{d}\right) (q+1)\right]^{ (1-1/p) \frac{\alpha/d}{z} } & = \left[ \frac{(q-p)(q+1)}{q(q+1)-p} \right]^{\frac{1}{p'} \frac{(q-p)(q+1)}{q(q+1)-p} }\left[ \frac{q(q+1)-p}{pq}\right]^{\frac{1}{p'} \frac{(q-p)(q+1)}{q(q+1)-p} },
\end{align*}
and that, by~\eqref{alphadqz} and since
\[
2 \frac{q}{p} \geq \frac{q(q+1)-p}{pq} \geq \frac{q}{p} \geq 1,
\]
one gets
\begin{align*}
\left[ \left(\frac{\alpha}{d}\right) (q+1)\right]^{ (1-1/p) \frac{\alpha/d}{z} } \leq 2  \left( \frac{q}{p}\right)^{1/p'}.
\end{align*}
This proves that $M_2^\theta \leq 2 \, d^\theta \, (q/p)^{1-1/p} \alpha^{-\theta}$.

Putting everything together, we proved that
\[
M_0^{1/q} M_1^{1-\theta} M_2^{\theta} \leq   2\,d \, \alpha^{-1}  (\e \, q+C(p,q))^{1/q} (q/p)^{1-1/p}.
\]
It remains to estimate the term in the parenthesis in the right hand side. Observe first that
\[
 (\e q+C(p,q))^{1/q}   \leq (\e\, q)^{1/q} + C(p,q)^{1/q} \leq 2\e + C(p,q)^{1/q},
\]
and then that
\[
 C(p,q)^{1/q} \leq \Big(1+\frac{p'}{q}\Big)^{1/q} =
 \frac{d-\alpha}{d} \, p' \, \Big(1+\frac{p'}{q}\Big)^{1/q-1} \leq  \frac{d-\alpha}{d} \, p' .
\]
After observing that $ (d-\alpha) \, p' /d  \geq 1$, the proof of~\eqref{Ipq} is
complete. This implies~\eqref{claim} and completes the proof.
\end{proof}

\section{Comparison with the Euclidean case}\label{sec:Eucl}

In this section we compare our embedding constant $A_1S(p,q)$ with the known embedding constant in the Euclidean case. As a preliminary remark, observe that if $G$ has polynomial growth, then $\delta=1$, and $\Ls =
\Delta$ is the sum-of-squares sub-Laplacian associated with ${\bf X}$.
Since the exponential dimension $D$ can be taken arbitrarily small,
one obtains $\tau_\delta =1$. Thus, in this case the Sobolev norm $\|
\cdot\|_{L^p_\alpha(\lambda)}$ is the graph norm of $(I
+\Delta)^{\alpha/2}$ in $L^p(\lambda)$. 

This in particular holds in $\R^d$, where ${\bf X} = \{\partial_1,
\dots, \partial_d\}$, $\Delta$ is the Laplacian, $\lambda$ is the
Lebesgue measure and $L^p_\alpha =L^p_\alpha(\lambda)$ is the
classical inhomogeneous Sobolev space. Theorem~\ref{teo:embed} in the Euclidean setting then
reads as 
\[
 \| f\|_{L^q}\leq A_1\,  S(p,q)    \|f\|_{L^p_\alpha },
\]
where $A_1$ depends only on the dimension $d$.

\bigskip

Let $0<\alpha<d$ and $p,q\in(1,\infty)$ be such that $\frac 1q=\frac 1p -\frac\alpha d$. Denote respectively by $E(p,q,d)$ and $E_H(p,q,d)$ the best embedding constants of $L^p_\alpha$ into $L^q$, and of 
$\dot{L}^p_\alpha$ into $L^q$, where $\dot{L}^p_\alpha$ is the homogeneous Sobolev space given by the
closure of the Schwartz functions with respect to the norm $\|f\|_{\dot{L}^p_\alpha} = \|\Delta^{\alpha/2}f\|_{L^p}$. Equivalently, $E(p,q,d)$ and $E_H(p,q,d)$  are respectively  the infimum of the constants $C_I,C_H>0$ such that
$$
\| (I+\Delta)^{-\alpha/2}  f\|_{L^q} \le C_I \|f\|_{L^p} 
\qquad
\text{and}
\qquad
\|  \Delta^{-\alpha/2} f\|_{L^q} \le C_H
\|f\|_{L^p}.
$$
 Now,   $E_H(p,q,d)$  equals 
\begin{equation}\label{ineq:Riesz-Bessel0}
E_H(p,q,d) =
\frac{1}{(2\pi)^\alpha}\frac{\Gamma((d-\alpha)/2)}{\Gamma(\alpha/2)}
{C_L(p,q,d)}
\end{equation}
where {$C_L(p,q,d)$} is the best constant for the
Hardy--Littlewood--Sobolev inequality,  which by~\cite[Theorem 4.3]{LiebLoss} can be estimated as follows:
\begin{equation}\label{stimaLL}
C_L(p,q,d)
 \le  \frac{d}{\alpha} \bigg(\frac{\omega_{d-1}}{d}\bigg)^{1-\frac{\alpha}{d}} 
\bigg(1-\frac{\alpha}{d} \bigg)^{1-\frac{\alpha}{d}}
\frac{1}{pq'} \bigg( {p'}^{\frac{1}{p'}+\frac{1}{q}} +q^{\frac{1}{p'}+\frac{1}{q}}\bigg),
\end{equation}
where $\omega_{d-1}$ is the surface measure of the unit sphere in  $\bbR^d$. Notice that $\frac{1}{p'}+ \frac{1}{q} = 1- \frac{\alpha}{d}$.  
In other words, by~\eqref{ineq:Riesz-Bessel0} and~\eqref{stimaLL} the best known bound for $E_H(p,q,d) $ is given by $E_H(p,q,d)\leq \widetilde E_H(p,q,d) $, where
\[
\widetilde E_H(p,q,d) =
\frac{1}{(2\pi)^\alpha}\frac{\Gamma((d-\alpha)/2)}{\Gamma(\alpha/2)}
\frac{d}{\alpha} \bigg(\frac{\omega_{d-1}}{d}\bigg)^{1-\frac{\alpha}{d}} 
\bigg(1-\frac{\alpha}{d} \bigg)^{1-\frac{\alpha}{d}}
\frac{1}{pq'} \bigg( {p'}^{\frac{1}{p'}+\frac{1}{q}} +q^{\frac{1}{p'}+\frac{1}{q}}\bigg).
\]
It turns out that $E(p,q,d)=E_H(p,q,d)$, so that
the best known bound for $E(p,q,d)$ is given in terms of $\widetilde
E_H(p,q,d)$; more precisely, we have the following result. For $p,q\in(1,\infty)$, $q\geq p$, set
\begin{equation}\label{Fpq}
F(p,q):= \frac{1}{\frac{1}{p'}+\frac{1}{q}} \frac{1}{pq'} \Big({p'}^\frac1q +q^{\frac{1}{p'}} \Big) .
\end{equation}
\begin{proposition}\label{prop:Lieb}
For all $p\in(1,\infty)$, $\alpha\in [0,d/p)$ and $q\in [p,\infty)$ such that $ \frac{1}{q}=\frac{1}{p}-
\frac{\alpha}{d}$,
\begin{equation}\label{ineq:Epq}
E(p,q,d) = E_H(p,q,d) \le \widetilde E_H(p,q,d).
\end{equation}
Moreover,  there exists a positive constant $B_1$ depending only on $d$ such  that,  for all $p,q,\alpha$ as above,
\begin{equation}\label{ineq:S-SH}
B_1^{-1} F(p,q) \le \widetilde E_H(p,q,d)\le B_1 F(p,q). 
\end{equation}
\end{proposition}
\begin{proof}
The equality in~\eqref{ineq:Epq} follows by observing that the best constant in the inequality $\|f\|_q \leq C \|(a + \Delta)^{\alpha/2}f\|_p$, with $f$ Schwartz, does not depend on $a> 0$ by rescaling; and then by using a limit argument (we thank one of the anonymous referees for pointing this out to us). The inequality in~\eqref{ineq:Epq} follows instead from the discussion preceding the proposition.

We now prove~\eqref{ineq:S-SH}. Using the conditions $0<\alpha<d$ and
$\frac1q=\frac1p-\frac{\alpha}{d}$, we have
\begin{align*}
  \frac{1}{(2\pi)^\alpha}\frac{\Gamma((d-\alpha)/2)}{\Gamma(\alpha/2)}
\frac{d}{\alpha}  &\bigg(\frac{\omega_{d-1}}{d}\bigg)^{1-\frac{\alpha}{d}} 
\bigg(1-\frac{\alpha}{d} \bigg)^{1-\frac{\alpha}{d}} \\
  & = \frac{1}{(2\pi)^\alpha}
   \frac{\Gamma(1+(d-\alpha)/2)}{\Gamma(1+\alpha/2)}
    \frac{1}{1-\frac{\alpha}{d}}
\bigg(\frac{\omega_{d-1}}{d}\bigg)^{1-\frac{\alpha}{d}} 
\bigg(1-\frac{\alpha}{d} \bigg)^{1-\frac{\alpha}{d}}
  \\
  & = \frac{1}{(2\pi)^\alpha}
   \frac{\Gamma(1+(d-\alpha)/2)}{\Gamma(1+\alpha/2)}
\bigg(\frac{\omega_{d-1}}{d}\bigg)^{1-\frac{\alpha}{d}} 
\bigg(1-\frac{\alpha}{d} \bigg)^{1-\frac{\alpha}{d}} \frac{1}{\frac{1}{p'}+\frac1q},
\end{align*}
and 
$$
\frac{1}{B(d)} \le \frac{1}{(2\pi)^\alpha}
   \frac{\Gamma(1+(d-\alpha)/2)}{\Gamma(1+\alpha/2)}
\bigg(\frac{\omega_{d-1}}{d}\bigg)^{1-\frac{\alpha}{d}} 
\bigg(1-\frac{\alpha}{d} \bigg)^{1-\frac{\alpha}{d}} \le B(d),
$$
where $B(d)$ is a constant depending only on $d$; observe indeed that each factor in the product above is bounded from above and below by a constant that depends only on $d$. Hence, 
\begin{align*}
 \frac{1}{B(d)} \frac{1}{\frac{1}{p'}+\frac{1}{q}} \frac{1}{pq'} \Big( {p'}^{\frac1q}+q^{\frac1{p'}} \Big) \le  \widetilde E_H(p,q,d) & \le  B(d)e^{1/e}
\frac{1}{\frac{1}{p'}+\frac{1}{q}} \frac{1}{pq'} \Big({p'}^{\frac1q}+q^{\frac1{p'}} \Big) ,
\end{align*}
since $1\le x^{1/x}\le e^{1/e}$ when $x\ge1$.
Hence, \eqref{ineq:S-SH} follows.
\end{proof}

We now show that similar estimates hold in our case, namely that the constant $S(p,q)$ is comparable to $\widetilde E_H(p,q,d)$, up to a constant depending only on $d$. In other words, we show that we recover the best known result, in terms of dependence on $p$ and $q$, when $G$ is a Euclidean space.

\begin{theorem}
There exists a constant $B_3$, depending only on $d$, such that
for all $p\in(1,\infty)$, $\alpha\in [0,d/p)$ and $q\in [p,\infty)$ such that $
\frac{1}{q }=\frac{1}{p} - \frac{\alpha}{d}$ 
we have
\begin{equation}\label{ineq:ok}
\frac{1}{B_3}S(p,q) \le    \widetilde E_H(p,q,d)  \le B_3  S(p,q).
\end{equation}
\end{theorem}

\begin{proof}
We are going  to show that
$S(p,q)$ is bounded above and below by absolute constants times
$F(p,q) $, and in view of~\eqref{ineq:S-SH} this will suffice. Since $F(p,q) = F(q',p') $ and $ S(p,q) = S(q',p')$,  it suffices in turn to consider the case $q\ge p'$.

We claim that in this regime
\begin{align*}
  \frac14 Q(p,q)\le F(p,q) \le 4
  Q(p,q),
\end{align*}
where $Q(p,q)=\frac{q^{1/p'}}{p-1}$.
Since $q\ge p'$, we also have $\frac{1}{p'}  \ge \frac1q$ and ${p'}^{\frac1q}\le q^{\frac{1}{p'}}$
(since $x\mapsto x^x$ is increasing on $[1,\infty)$).
Then, since as before $1\le q'\le2$,  
$$
F(p,q) \le 2 \frac{1}{\frac{1}{p'}pq'} q^{\frac{1}{p'}}=
\frac{2}{q'(p-1)}q^{\frac{1}{p'}}\le 2\frac{q^{\frac{1}{p'}}}{p-1}
= 2Q(p,q).
$$
On the other hand,
  $$
 F(p,q) \ge \frac{1}{\frac{2}{p'}q'p}q^{\frac{1}{p'}} =
 \frac{q^{\frac{1}{p'}}}{2q'(p-1)}
 \ge \frac14 Q(p,q).
 $$
 This proves the claim. It remains to show that if $q\geq p'$ then $S(p,q) = Q(p,q)$, namely $ Q(p,q)\le Q(q',p')$. The latter inequality is 
$$
 \frac{q^{\frac{1}{p'}}}{p-1}\le  \frac{{p'}^{\frac{1}{q}}}{q'-1}.
$$
Multiplying both
sides by $pq'$, it becomes
$$
q'p'^{\frac{1}{q'}} \le p q^{\frac1p}  .
$$
Since $q\ge p'$, hence $q'\le p$, it suffices to show that
$p'^{\frac{1}{q'}} \le  q^{\frac1p}$, that is, ${p'}^p \le
q^{q'}$. But this follows since $p'\le q$ and the
function $x\mapsto e^{\frac{x}{x-1}\log x}$ is increasing in $[1,\infty)$. This concludes the proof.
\end{proof}

\section{A Moser--Trudinger inequality}\label{Sec_MT}
As an application of Theorem~\ref{teo:embed}, we shall prove a quantitative Moser--Trudinger inequality.  To do this, we will need a precise version of the interpolation inequality~\cite[eq.\ (6.1)]{BPV1} associated to the interpolation space  $(L^p(\lambda),L^p_{\alpha}(\lambda))_{[\theta]}=L^p_{\theta\alpha}(\lambda)$ with respect to the complex method. To prove this refined estimate, we follow some ideas developed in~\cite{AM}; see also~\cite{PV}.
\begin{proposition}\label{teo:interpolation}
  Let $p\in(1,\infty)$ and define
  \begin{equation*}
\mathcal C_p= \inf_{\sigma>0}   \sup_{t\in\bbR} \e^{\sigma(1- t^2)}  
    \|  (\tau_\delta I+\mathcal L)^{it}  \|_{L^p(\lambda)\rightarrow L^p(\lambda) } .
    \end{equation*}
    Then $1\leq \mathcal C_p<\infty$ and for all $f \in L^p_\alpha(\lambda)$, $\alpha\geq 0$, and $\theta\in(0,1)$ we have 
\begin{equation}\label{f:interpolationinequality}
\|f\|_{L^p_{\theta\alpha}(\lambda)}\leq  \mathcal C_p
\|f\|_{L^p(\lambda)}^{1-\theta} \,
\|f\|_{L^p_{\alpha}(\lambda)}^{\theta}\,.
\end{equation} 
\end{proposition}
  
\begin{proof} For $\sigma>0$, let
\[
\mathcal C_{p,\sigma} =   \sup_{t\in\bbR} \e^{\sigma(1- t^2)}  
    \|  (\tau_\delta I+\mathcal L)^{it}  \|_{L^p(\lambda)\rightarrow L^p(\lambda) }.
\]
Since $\mathcal C_{p,\sigma}$ is finite for all $\sigma>0$ by~\cite[Corollary 1]{Cowling}, see 
also~\cite{Meda}, it follows that $\mathcal C_{p}$ is finite. Moreover, since $(\tau_\delta I+\mathcal L)^{it} = I $ for $t=0$, one gets $\mathcal C_{p,\sigma} \geq \e^{\sigma} \geq 1$, hence also $\mathcal C_{p}\geq 1$.

  Suppose that $f=\sum_{j=1}^Na_j\chi_{E_j}$,
$h=\sum_{k=1}^{N'}a_k'\chi_{E'_k}$  are two simple functions on
$G$.  Let $S= \big\{ z\in\bbC:\, 0<\Re z<1\big\}$, and let $\ov S$
denote its closure. For every $z\in \ov S$ we define 
$$
w(z)=\e^{\sigma z^2}\int_G(\tau_\delta I+\mathcal L)^{-\alpha z/2}f(x)h(x)\, \dd\lambda(x).
$$
Then $w$ is holomorphic on $S$, continuous on $\ov S$ and $w$ is
bounded on $\ov S$. Indeed,  
$$
\begin{aligned}
\sup_{z\in\ov{S}} |w(z)|&\leq \sum_{j=1}^N\sum_{k=1}^{N'}|a_j||a_k'|
\sup_{z\in\ov{S}}   \Big| \e^{\sigma z^2}\int_{E_k'}(\tau_\delta I+\mathcal
L)^{-\alpha z/2}\chi_{E_j}(x)\, \dd\lambda(x) \Big|\\ 
&\leq \mathcal C_{p,\sigma} \sum_{j=1}^N\sum_{k=1}^{N'}|a_j||a_{k}'|\lambda
(E_{k}')^{1/p'} \sup_{0\leq x\leq 1}\|(\tau_\delta I+\mathcal
L)^{-\alpha x/2}\|_{L^p(\lambda)\rightarrow
  L^p(\lambda)}\lambda(E_j)^{1/p} <\infty\,.
\end{aligned}
$$
We now observe that for every $t\in\mathbb R$
\begin{equation*}
|w(it)|\leq \mathcal C_{p,\sigma} \|f\|_{L^p(\lambda)}\|h\|_{L^{p'}(\lambda)}
\end{equation*}
and
\begin{equation*}
|w(1+it)|\leq  \mathcal C_{p,\sigma} \|(\tau_{\delta}I+\mathcal L)^{-\alpha/2}f\|_{L^p(\lambda)}\|h\|_{L^{p'}(\lambda)}\,.
\end{equation*}
 By the classical three lines theorem  it follows that
$$
|w(1-\theta)|\leq \mathcal C_{p,\sigma}   \|f\|^{\theta }_{L^p(\lambda)}
\|(\tau_{\delta}I+\mathcal L)^{-\alpha/2}f\|^{1-\theta
}_{L^p(\lambda)}\|h\|_{L^{p'}(\lambda)}\,. 
$$
By taking the supremum over all simple functions $h$ such that
$\|h\|_{L^{p'}(\lambda)}\leq 1$ we have 
$$
\|(\tau_\delta I+\mathcal L)^{- (1-\theta)\alpha/2}f
\|_{L^p(\lambda)}\leq \mathcal C_{p,\sigma}   \|f\|^{\theta }_{L^p(\lambda)}
\|(\tau_{\delta}I+\mathcal L)^{-\alpha/2}f\|^{1-\theta
}_{L^p(\lambda)}.  
$$
By using the density of simple functions in $L^p(\lambda)$ and
choosing $g=(\tau_{\delta}I+\mathcal L)^{-\alpha/2}f$ we get 
$$
\|(\tau_\delta I+\mathcal L)^{\theta \alpha/2}g \|_{L^p(\lambda)}\leq
\mathcal C_{p,\sigma}  \| (\tau_{\delta}I+\mathcal L)^{\alpha/2} g\|^{\theta
}_{L^p(\lambda)} \|g\|^{1-\theta }_{L^p(\lambda)},  
$$
which is equivalent to
$$
\|g\|_{L^p_{\theta\alpha}(\lambda)}\leq \mathcal \mathcal C_{p,\sigma}   \|   g\|^{\theta
}_{L^p_{\alpha}(\lambda)} \|g\|^{1-\theta }_{L^p(\lambda)}\,.
$$
By taking the infimum over all $\sigma>0$, the inequality \eqref{f:interpolationinequality} follows. 
 \end{proof}

As a corollary of the estimate of Theorem~\ref{teo:embed} and Proposition~\ref{teo:interpolation}, we
obtain the following global Moser--Trudinger inequality. 
Keeping the notation therein, we define
$$\gamma_1
=[\e\,   \big(\mathcal C_pA_1(p'-1\big)   )^{p'}p']^{-1}.
$$

\begin{theorem}\label{teo:GMT2}
Let $p\in (1,\infty)$. For $\gamma \in
[0,\gamma_1)$ and $f\in L^p_{d/p}(\lambda)$ with
$\|f\|_{L^p_{d/p}(\lambda)}\leq 1$, 
\begin{equation}\label{eq:GMT2}
\int_G \Big( \exp(\gamma |f|^{p'})-\sum_{0\leq k<p-1}
\frac{\gamma^k}{k!} |f|^{p'k} \Big)\, \dd \lambda \leq C(G,\mathbf{X},p) 
\|f\|_{L^p(\lambda)}^p. 
\end{equation}
\end{theorem}

 We point out that, even in the case
of the Laplacian in $\R^d$, the best constant $\gamma_1$ for which \eqref{eq:GMT2} holds is not known, other than
in the cases $d/p=1$~\cite{Li-Ruf} and $d/p=2$~\cite{Lam-Lu}.  
  
\begin{proof} 
By Theorem~\ref{teo:embed} and the
interpolation inequality~\eqref{f:interpolationinequality}, when $q>p$  we obtain
\begin{equation}\label{new}
\|f\|_{L^q(\lambda)} \leq  A_1 \,   S(p,q)   \,\mathcal C_p  \|f\|_{L^p_{d/p}(\lambda)}^{1-p/q}\|f\|_{L^p(\lambda)}^{p/q}.
\end{equation}
Then, if $\|f\|_{L^p_{d/p}(\lambda)}\leq 1$,
\begin{align}
\int_G \Big( \exp(\gamma |f|^{p'})-\sum_{0\leq k<p-1} \frac{\gamma^k}{k!} |f|^{p'k} \Big)\, \dd \lambda
&= \sum_{k\geq p-1}  \frac{\gamma^k}{k!} \|f\|_{L^{p'k}(\lambda)}^{p'k}\notag\\
  & \leq \|f\|_{L^p(\lambda)}^p \sum_{k\geq p-1}  \frac{\gamma^k}{k!} (\mathcal C_pA_1)^{p'k} S(p,p'k)^{p'k} \ .\label{new2} 
\end{align}
Observe that since $S(p,q) = Q(p,q)$ when $q\geq p'$ and $p'k \geq p'$, 
\begin{align*}
S(p,p'k)^{p'k}
  & = \min\bigg( \frac{(p'k)^{1/p'}}{p-1}, \frac{{p'}^{1/(p'k)}}{(p'k)'-1}\bigg)^{p'k} =\frac{(p'k)^k}{(p-1)^{p'k}}\,.
\end{align*}
Plugging this estimate into~\eqref{new2} we obtain 
\begin{align*}
\int_G \Big( \exp(\gamma |f|^{p'})-\sum_{0\leq k<p-1} \frac{\gamma^k}{k!} |f|^{p'k} \Big)\, \dd \lambda
& \le \|f\|_{L^p(\lambda)}^p \sum_{k\geq p-1}  \frac{\gamma^k}{k!} 
\big(\mathcal C_pA_1(p'-1) \big)^{p'k}(p'k)^k\\
  & \leq C(G, {\bf{X}},p) \|f\|_{L^p(\lambda)}^p
\end{align*}
if $\gamma<\gamma_1$. The proof of the theorem is complete.
\end{proof}

\section{The case of general measures}\label{Sec_gen}
In this final section we consider the case of more general sub-Laplacians and relatively invariant measures, as in~\cite{BPTV}, where different phenomena appear. We denote by $\rho$ the right Haar measure such that $\dd \lambda = \delta^{-1} \, \dd \rho$, and by $\chi$ a continuous positive
character of $G$. We then let $\mu_\chi$ be the measure with density $\chi$
with respect to $\rho$. As $\delta$ is a continuous positive character, $\mu_\delta = \lambda$. Since
\begin{equation*}
\sup\nolimits_{|x|\leq r} \chi(x) = \e^{\cf(\chi) r}, \quad
\mbox{where} \quad \cf(\chi) = ( |X_1\chi(e)|^2 + \cdots +
|X_\ell\chi(e)|^2)^{1/2},
\end{equation*}
cf.~\cite{HMM},  and $V(r) = \rho(B_r)$, the metric measure space $(G, d_C, \mu_\chi)$ is locally
doubling, though not doubling in general. 

The spaces $L^p(\mu_\chi)$ are defined classically and in the same way as the spaces $L^p(\lambda)$ described above. We denote by $\Delta_{\chi}$ the sub-Laplacian with drift 
\begin{equation*}
\Delta_\chi = - \sum_{j=1}^\ell (X_j^2 +(X_j\chi)(e) X_j),
\end{equation*}
and recall that it is symmetric on $L^2(\mu_\chi)$. Observe that $\Delta_\delta =\Ls$ and $\Delta_1$ is the standard
left-invariant sum-of-squares sub-Laplacian. The operator $\Delta_\chi$ generates a diffusion semigroup, namely $(\e^{-t\Delta_\chi})_{t>0}$ extends to a contraction semigroup on
$L^p(\mu_\chi)$ for every $p \in [1, \infty]$ whose
infinitesimal generator we still
denote by $\Delta_\chi$;  see~\cite{HMM, BPTV, BPV1,BPV2} for more on these matters.

When $p\in (1,\infty)$ and $\alpha> 0$, we define the Sobolev spaces $L^p_\alpha(\mu_\chi)$ as the space of functions $f\in L^p(\mu_\chi)$ such that $(\tau_\chi I +
\Delta_{\chi})^{\alpha/2} f\in L^p(\mu_\chi)$, endowed with the norm 
\[
\| f\|_{L^p_{\alpha}(\mu_\chi)} = \|(\tau_\chi I + \Delta_\chi)^{\alpha/2}f\|_{L^p(\mu_\chi)},
\]
where
\begin{equation}\label{fixedtranslationchi}
\tf_\chi = \max \left\{\frac{2}{b} \left[ \cf(\delta \chi^{-1}) +2D+b_0 \right]^2 - \frac{1}{4}\cf(\chi)^2, 1\right\}
\end{equation}
is the counterpart (or generalized version) of~\eqref{fixedtranslation}. Observe that $\cf(\delta\chi^{-1}) =0$ if $\chi = \delta$ or, equivalently, if $\mu_\chi = \lambda$, so our notation is coherent with the one used in previous sections.

\smallskip

We recall from~\cite{BPTV} that an embedding as the one of Theorem~\ref{teo:embed} fails if $\lambda$ is replaced by any other measure $\mu_\chi$; and as we show below in Remark~\ref{rem:muchi}, a global Moser--Trudinger inequality as Theorem~\ref{teo:GMT2} also does not hold if $\mu_\chi\neq \lambda$. Nevertheless, we can prove an alternative version of Sobolev embedding, and a local Moser--Trudinger {inequality} (that is, for compactly
supported functions). We shall first need to extend some definitions and results, given above in the case of the left measure $\lambda$, to the case of $\mu_\chi$.

We denote by $p_t^\chi$
the convolution kernel of $\e^{-t\Delta_\chi}$, and we recall that by
~\cite[Theorem IX.1.3]{VCS}, equivalently~\eqref{heatkernelestimate}, and~\cite[eq.\ (2.8)]{BPTV},
\begin{equation}\label{heatkernelestimatechi}
p_t^\chi(x)  \leq c \, (\delta \chi^{-1})^{1/2}(x)\, (1\wedge t)^{-\frac{d}{2}} \, \e^{-\frac{1}{4} t \cf(\chi)^2}\,  \e^{-b \frac{|x|^2}{t}}, \qquad x\in G,\,t>0
\end{equation}
where $b$ and $c$ are those of~\eqref{heatkernelestimate}.

For $\alpha>0$, let $G_{\chi}^{\alpha}$ be the convolution kernel of $(\tau_\chi I +\Delta_\chi)^{-\alpha/2}$, and define $ G_{\chi}^{\alpha, \loc} =G_{\chi}^{\alpha}\mathbf{1}_{B_1}$ and  $G_{\chi}^{\alpha, \glob} = G_{\chi}^{\alpha}\mathbf{1}_{B_1^c}$. The following result can be proved exactly in the same way as Lemma~\ref{Lemma4.1-revised}, and its proof is omitted.

\begin{lemma}\label{Lemma4.1-revisedchi}
  There exists $C=C(G, \mathbf{X})>0$ such
  that, for $\alpha\in (0,d)$ and $x\in G$,
\begin{align*}
|G_{\chi}^{\alpha, \loc} (x)| &\leq C\, \frac{\alpha}{d-\alpha} (\delta \chi^{-1})^{1/2}(x) |x|^{\alpha-d}\mathbf{1}_{B(e,1)}(x) ,\\
|G_{\chi}^{\alpha, \glob}(x)| &\leq  C \, (\delta \chi^{-1})^{1/2}(x)\, \e^{-(2D+\cf(\delta \chi^{-1})+b_0)|x|}\mathbf{1}_{B(e,1)^c}(x).
\end{align*}
  \end{lemma}
Define now $\sf(\chi)= \max_{B_1} \chi\delta^{-1} = \e^{\cf(\chi\delta^{-1})}$, and observe that $\sf(\chi) \geq 1$ for all $\chi$'s.

\begin{proposition}\label{l: embeddingmuchi}
Let $p\in (1,\infty)$ and $q\in [p,\infty)$. There exists  $A_2=A_2(G, \bX)>0$ such that 
\begin{equation}\label{embedding}
  \| f\|_{L^q(\mu_{\chi^{q/p}\delta^{1-q/p}})} \leq
  \frac{A_2 \, \sf(\chi) }{p-1} \left(1+\frac{q}{p'}\right)^{\frac{1}{q}+\frac{1}{p'}} \| f\|_{ L^{p}_{d/p}(\mu_\chi)}
\end{equation}
for all $f\in L^{p}_{d/p}(\mu_\chi)$.
\end{proposition}
\begin{proof}
By Young's inequality~\eqref{Young}, we obtain that 
\begin{align}
&  \|(\tau_\chi I+\Delta_\chi)^{-d/2p}g\|_{L^q(\mu_{\chi^{q/p}\delta^{1-q/p}})}  \notag\\
  & \quad= \|(\chi\delta^{-1})^{1/p}g \ast (\chi\delta^{-1})^{1/p} G_\chi^{d/p}\|_{L^q(\lambda)} \nonumber \\
 & \quad\leq  \|(\chi\delta^{-1})^{1/p} g\|_{L^p(\lambda)}  \|
   (\chi^{-1}\delta)^{1/p} \widecheck{G}_\chi^{d/p}\|_{L^r(\lambda)}^{r/{p'}} \,\|
    (\chi\delta^{-1})^{1/p} G_\chi^{d/p}\|_{L^r(\lambda)}^{r/{q}} \nonumber \\
  & \quad= \| g\|_{L^p(\mu_\chi)}  \| (\chi^{-1}\delta)^{1/p}
    \widecheck{G}_\chi^{d/p}\|_{L^r(\lambda)}^{r/{p'}} \,\| (\chi\delta^{-1})^{1/p} G_\chi^{d/p}\|_{L^r(\lambda)}^{r/{q}},  \label{appl-Young}
\end{align}
where
$r\in (1,\infty)$ is such that
$\frac1p+\frac1r=1+\frac1q$. We split $G_\chi^{d/p}$ into $ G_{\chi}^{d/p, \loc}$ and $G_{\chi}^{d/p, \glob}$, and estimate the integrals of the two terms separately.

By Lemma~\ref{Lemma4.1-revisedchi}, we obtain
\begin{align*}
  \| (\chi\delta^{-1})^{1/p} G_\chi^{d/p,\loc}\|_{L^r(\lambda)} 
  & \le \frac{C}{p-1}\,  \Big(\sum_{k=0}^\infty \int_{2^{-k-1}<|x|\le 2^{-k}} (\delta \chi^{-1})^{r(\frac12-\frac1p)}(x)    |x|^{r(d/p-d)} \,\dd\lambda(x)\Big)^{1/r}\\
  &\leq \frac{C}{p-1}\, \sf(\chi) \Big(\sum_{k=0}^\infty 2^{-kr(d/p-d)-kd}\Big)^{1/r} \\
  & \leq \frac{C}{p-1}\, \sf(\chi)  \Big(\int_0^1 u^{(d/p-d)r} u^{d-1}\, \dd u\Big)^{1/r} = \frac{C\, \sf(\chi) }{p-1}\,\left(1+ \frac{q}{p'}\right)^{\frac{1}{q}+\frac{1}{p'}},
\end{align*}
where we used that
\begin{align}\label{supmodulo}
\sup_{y\leq |x| } (\delta \chi^{-1})^{1/2-1/p}(y) = \sup_{y\leq |x|} (\delta \chi^{-1})^{|1/2-1/p|}(x) = \e^{\cf(\chi \delta^{-1})|x|},
\end{align}
and that $|1/2-1/p|\leq 1$.

As for the global part of the kernel, using again~\eqref{supmodulo},
\begin{align}    
  \| (\chi\delta^{-1})^{1/p} G_\chi^{d/p,\glob}\|_{L^r(\lambda)}   
  &\leq C  \Big(\int_0^\infty (\chi\delta^{-1})^{r(1/p-1/2)} \e^{-r(2D+\cf(\chi \delta^{-1}) +b_0)|x|} \,\dd\lambda \Big)^{1/r}\notag \\
 & \leq C  \Big(\int_0^\infty  \e^{-r(2D+b_0)|x|} \,\dd\lambda \Big)^{1/r}\notag \\
  & \le C\,   \Big(\sum_{k=0}^\infty  \e^{-r(2D+b_0) 2^k+D2^{k+1}} \Big)^{1/r}\leq C .
  \label{est-G-dp}
\end{align}
The term $\| \widecheck{G}_{d/p}^c\|_{L^r(\lambda)}$ can be estimated in the same way, in view of~\eqref{supmodulo} and by the radiality of the other terms appearing in the bound of  Lemma~\ref{Lemma4.1-revisedchi}.
\end{proof}

Keeping the notation of Proposition~\ref{l: embeddingmuchi}, for $1<p<\infty$ we define 
\[
\gamma_2= \left[ \e \, \left(\frac{A_2 \sf(\chi)^{2}}{p-1}\right)^{p'} \right]^{-1}.
\]
 The following result is inspired by~\cite{Str}. 
 \begin{theorem}\label{teo:LMT}
Let $p\in (1,\infty)$. For $\gamma \in [0,\gamma_2)$,
\[
\sup_{\| f\|_{L^p_{d/p}(\mu_\chi)}\leq 1,\: \supp f\subseteq B(e,1) } \int_G \left( \exp(\gamma |f|^{p'})-1 \right)\, \dd \mu_\chi <\infty.
\]
\end{theorem}

\begin{proof}
We first notice that if $f$ is supported in $B_1$ and $q> p$, then
\[
\| f\|_{L^q(\mu_\chi)}  = \| (\chi \delta^{-1})^{\frac{1}{q}-\frac{1}{p}} f\|_{L^q(\mu_{\chi^{q/p}\delta^{1-q/p}})} \leq \sf(\chi )\| f\|_{L^q(\mu_{\chi^{q/p}\delta^{1-q/p}})} ,
\] 
so by Proposition~\ref{l: embeddingmuchi} 
\begin{align}
\| f\|_{L^q(\mu_\chi)} & \leq \frac{A_2\, \sf(\chi)^2}{p-1} \left(1+\frac{q}{p'}\right)^{\frac{1}{q}+\frac{1}{p'}} \| f\|_{ L^{p}_{d/p}(\mu_\chi)}.\label{Eq1}
\end{align}
If $f$ is supported in $B_1$ and $\|f\|_{L^p_{d/p}(\mu_\chi)} \leq
1$, then 
\[
\|f\|_{L^p(\mu_\chi)} \leq \|(\tau_\chi I+\Delta_\chi)^{-d/{2p}} \|_{L^p(\mu_\chi)\rightarrow L^p(\mu_\chi)}=C(\chi,p),
\]
and
\begin{align*}
\int_G &\left( \exp(\gamma |f|^{p'})-1 \right)\, \dd \mu_\chi= \sum_{k=1}^\infty \frac{\gamma^k}{k!} \|f\|_{L^{p'k}(\mu_\chi)}^{p'k} \\
& \leq C(\chi,p) \sum_{1\leq k<p/p'} \frac{\gamma^k}{k!} \mu_\chi(B(e,1))^{1- k(p'-1)}  +  \sum_{k\geq p/p'} \frac{\gamma^k}{k!} \left( \frac{A_2\, \sf(\chi)^2}{p-1}\right)^{p'k} (k+1)^{k+1}\,,
\end{align*}
where we applied~\eqref{Eq1} when $kp'\geq p$, and H\"older's
inequality and the support condition of $f$ if $kp'< p$. If $\gamma\in [0,\gamma_2)$, then the latter series is convergent and the theorem is proved. 
\end{proof}

\begin{remark}\label{rem:muchi}
Theorem~\ref{teo:GMT2} does not hold with any other $\mu_\chi$ in place of $\lambda$. Indeed, if there exist $p\in (1,\infty)$, $C>0$ and $\gamma>0$ such that for all $f\in L^p_{d/p}(\mu_\chi)$, $\|f\|_{L^p_{d/p}(\mu_\chi)}\leq 1$,
\begin{equation}\label{disprove}
\int_G \Big( \exp(\gamma |f|^{p'})-\sum_{0\leq k<p-1} \frac{\gamma^k}{k!} |f|^{p'k} \Big)\, \dd \mu_\chi \leq C\|f\|_{L^p(\mu_\chi)}^p,
\end{equation}
then necessarily $\mu_\chi = \lambda$.

To see this, assume that~\eqref{disprove} holds for all $f\in
L^p_{d/p}(\mu_\chi)$, $\|f\|_{L^p_{d/p}(\mu_\chi)}\leq 1$, with
$\mu_\chi \neq \lambda$, i.e.\ $\chi \neq \delta$. We first prove that
then~\eqref{disprove} holds for all $f\in L^p_{d/p}(\mu_\chi)$, with
no restriction on its norm (other than being finite). Recall, indeed,
that for any $y\in G$ and $f\in  L^p_{d/p}(\mu_\chi)$, denoting by
$L_y$ the left translation by $y\in G$, one has
\[
\| L_y f\|_{ L^p_{d/p}(\mu_\chi)} = (\chi \delta^{-1})^{1/p}(y) \|f\|_{ L^p_{d/p}(\mu_\chi)}.
\]
Since $(\chi \delta^{-1})^{-1/p}$ is a positive nonconstant character, it is unbounded; thus there exists $y\in G$ such that
\[
(\chi \delta^{-1})^{-1/p}(y) \geq \|f\|_{ L^p_{d/p}(\mu_\chi)}.
\] 
Equivalently, $(\chi \delta^{-1})^{1/p}(y) \|f\|_{ L^p_{d/p}(\mu_\chi)} \leq 1$, hence  $\| L_y f\|_{ L^p_{d/p}(\mu_\chi)}\leq 1$. Thus, we may apply~\eqref{disprove} to $L_yf$; and by a change of variable, one obtains~\eqref{disprove} for $f$ where the constant $C$  does not depend on the norm of $f$.

But~\eqref{disprove} cannot hold without restriction on the norm of $f\in L^p_{d/p}(\mu_\chi)$. Indeed, let $\sigma\geq 1$ and consider $\sigma f$, which still belongs to $L^p_{d/p}(\mu_\chi)$ for any $\sigma$. Then, by~\eqref{disprove} applied to $\sigma f$,
\begin{align*}
 \int_ G \sum_{k\geq p-1} \frac{\gamma^k}{k!} \sigma^{p'k} |f|^{p'k} \, \dd \mu_\chi 
 & \leq C \, \sigma^p \|f\|_{L^p(\mu_\chi)}^p.
\end{align*}
Since
\[ \int_ G \sum_{k\geq p-1} \frac{\gamma^k}{k!} \sigma^{p'k} |f|^{p'k} \, \dd \mu_\chi  \geq \int_ G \sum_{k\geq p} \frac{\gamma^k}{k!} \sigma^{p'k} |f|^{p'k} \, \dd \mu_\chi  \geq \sigma^{pp'} \int_ G \sum_{k\geq p} \frac{\gamma^k}{k!} |f|^{p'k} \, \dd \mu_\chi, 
\]
one obtains
\[
\sigma^{p(p'-1)} \int_ G \sum_{k\geq p} \frac{\gamma^k}{k!} |f|^{p'k} \, \dd \mu_\chi \leq C \|f\|_{L^p(\mu_\chi)}^p
\]
for all $\sigma\geq 1$, which is a contradiction since $p(p'-1)>0$.
\end{remark}

\end{document}